\documentclass[11pt]{amsart}

\usepackage{amssymb}
\usepackage{amsmath}
\usepackage{amsfonts}
\usepackage{graphicx}
\usepackage{amsthm}
\usepackage{enumerate}
\usepackage[mathscr]{eucal}
\usepackage{mathrsfs}
\usepackage{verbatim}
\usepackage{yhmath}

\usepackage{esint}

\usepackage{xcolor}

\makeatletter
\@namedef{subjclassname@2020}{%
  \textup{2020} Mathematics Subject Classification}
\makeatother

\numberwithin{equation}{section}
\numberwithin{figure}{section}

\theoremstyle{plain}
\newtheorem{theorem}{Theorem}[section]
\newtheorem{lemma}[theorem]{Lemma}

\theoremstyle{plain}

\theoremstyle{remark}

\newtheorem{remarks}{Remarks}

\DeclareMathOperator{\supp}{supp}

\allowdisplaybreaks[4]

\begin{document}

\title[two-point polynomial patterns]{two-point polynomial patterns in  subsets of positive density in $\mathbb{R}^n$}

\author{Xuezhi Chen}
\address{Institute of Applied Physics \& Computational Mathematics\\
	Beijing, 100088\\ P.R. China}
\email{xuezhi-chen@foxmail.com}

\author{Changxing Miao}
\address{Institute of Applied Physics \& Computational Mathematics\\
	Beijing, 100088\\ P.R. China}
\email{miao\_changxing@iapcm.ac.cn}




\begin{abstract}
	Let $\gamma(t)=(P_1(t),\ldots,P_n(t))$ where $P_i$ is a real polynomial with zero constant term for each $1\leq i\leq n$.  We will show the existence of the configuration $\{x,x+\gamma(t)\}$ in sets of positive density $\epsilon$ in $[0,1]^n$ with a gap estimate $t\geq \delta(\epsilon)$ when $P_i$'s are arbitrary,
	and in $[0,N]^n$ with a gap estimate $t\geq \delta(\epsilon)N^n$ when $P_i$'s are of distinct degrees where $\delta(\epsilon)=\exp\left(-\exp\left(c\epsilon^{-4}\right)\right)$ and $c$ only depends on $\gamma$. To prove these two results, decay estimates of certain oscillatory integral operators and Bourgain's reduction are primarily utilised. For the first result, dimension-reducing arguments are also required to handle the linear dependency. For the second one, we will prove a stronger result instead, since then an anisotropic rescaling is allowed in the proof to eliminate the dependence of the decay estimate on $N$. And as a byproduct, using the strategy token to prove the latter case, we extend the corner-type Roth theorem previously proven by the first author and Guo.
\end{abstract}

\maketitle

\section{Introduction}\label{intro}

It was conjectured by Lov\'asz that a set $A\subset\{1,2,\ldots,N\}$ with no non-zero square differences has size $o(N)$. Furstenberg \cite{furstenberg77} proved this conjecture using ergodic theory. Almost at the same time, S\'ark\"ozy \cite{sarkozy78} obtained a stronger result for such an $A$ with a quantitative bound 
\begin{equation}\label{s-bound}
	|A|=O\left(\frac{N(\log\log N)^{2/3}}{(\log N)^{1/3}}\right)=o(N)
\end{equation}  
using the Hardy-Littlewood circle method  elaborated by Roth \cite{roth53}. We refer this result as  the Furstenberg-S\'ark\"ozy theorem. Pintz, Steiger and Szemer\'edi \cite{pss88} improved the bound \eqref{s-bound} and 
recently Bloom and Maynard \cite{bm22} got the best upper bound so far, saying
\begin{equation}\label{bm-bound}
	|A|=O\left(\frac{N}{(\log N)^{c\log\log\log N}}\right)
\end{equation}
for some absolute constant $c>0$.

Besides strengthening the upper bound, it is also an interesting problem to extend the square differences to  more general differences. For studies in this area, one can refer to S\'ark\"ozy \cite{sarkozy78-2}, Balog, Pelik\'an, Pintz and Szemer\'edi \cite{bpps94}, Rice \cite{rice19}, etc. By combining the methods in \cite{rice19} and \cite{bm22}, Arala \cite{arala23} gained the Bloom-Maynard bound \eqref{bm-bound} for arbitrary intersective polynomial difference\footnote{A polynomial $P(t)\in\mathbb{Z}[t]$ is called intersective if for any positive $q$, the  congruence $P(n)\equiv0 \,(\textrm{mod}\,q)$ has an integer solution.} with the constant $c$ only depending on the degree of the given polynomial  .

Note that for a set $A\subset\{1,2,\ldots,N\}$, considering whether $A$ has non-zero square differences, i.e. $(A-A)\cap (\{t^2:t\in\mathbb{Z}\}\setminus\{0\})\neq\emptyset$, is equivalent to searching for the patterns $\{x,x+t^2\}$ in $A$ with $t\neq 0$. The Furstenberg-S\'ark\"ozy theorem can be viewed as a special case (i.e. of length 2) of the polynomial Szemer\'edi theorem proven by Bergelson and Leibman \cite{bl96} using ergodic theory. However, using their result, one can only get the $o(N)$ bound as in \cite{furstenberg77} for the special pattern $\{x,x+t^2\}$. Peluse and Prendiville \cite{pp22} derived that a subset of $\{1,2,\ldots,N\}$ lacking of nontrivial triples $\{x,x+t,x+t^2\}$  has size
\begin{equation}\label{pp-bound}
	O\left(\frac{N}{(\log N)^c}\right)
\end{equation}
with $c=2^{-150}$, 
which is the best bound at present but still worse than the bound in \eqref{bm-bound}. It is natural that for the shorter pattern $\{x,x+t^2\}$, Bloom and Maynard \cite{bm22} were able to obtain a better upper bound than \eqref{pp-bound}.

Kuca, Orponen and Sahlsten \cite{kos23} considered a continuous analogue of this problem for sets of fractional dimension. They proved that there exists an absolute constant $\epsilon>0$ such that if $K\subset\mathbb{R}^2$ is a compact set with Hausdorff dimension $\dim_{\mathcal{H}}K\geq 2-\epsilon$, then there exist $x\in K$ and $z\neq 0$ such that $x+(z,z^2)\in K$.  Bruce and Pramanik \cite{bp23} extended this result to more general curves in higher dimensions. 

In the Euclidean setting, let $\gamma:\mathbb{R}\rightarrow\mathbb{R}^n$ be a continuous function with  $\gamma(0)=0$. Assume that $E\subset\mathbb{R}^n$ has positive Lebesgue measure. It is trivial to show that $E$ contains the pattern $\{x,x+\gamma(t)\}$ for some $t\neq 0$ or  more general finite patterns. Indeed, we have
\begin{equation*}
	\int_{\mathbb{R}^n}\!\textbf{1}_E(x)\textbf{1}_E(x+\gamma(t))\,\mathrm{d}x>0
\end{equation*}
holds for all sufficiently small $t$ by the continuity. Bourgain \cite{bou88} consider the existence of the patterns $\{x,x+t,x+t^d\}$ for fixed $d\geq 2$ with an explicit gap estimate of $t$ in the positive density subset of $[0,N]$. Later Durcik, Guo and Roos \cite{dgr19} and the first author, Guo and Li \cite{cgl21} extended Bourgain's result to general polynomial patterns and recently Krause, Mirek, Peluse and Wright \cite{kmpw22} studied the polynomial Szemer\'edi-type problem in topological fields. For higher dimensions, Christ, Durcik and Roos \cite{cdr21}, and the first author and Guo \cite{cg23} studied the corner-type configurations of the form $\{(x,y),(x+P_1(t),y),(x,y+P_2(t))\}$ in subsets of $[0,1]^2$ in the plane. And Durcik,  Kova\v{c} and Stip\v{c}i\'c \cite{dks23} proved the existence of $\{(x,y),(x+t,y+at^{\beta})\}$ in a positive measure subset in the plane with  $a$  in a whole interval $I$ and a uniform point $(x,y)$ for positive $\beta\neq 1$. 

In this paper, inspired by the comments in \cite{dks23}, we first consider a continuous variant of the Furstenberg-S\'ark\"ozy theorem in the unit cube of $\mathbb{R}^n$ for general polynomial curves. Throughout this paper, for each $1\leq i\leq n $, we let $P_i:\mathbb{R}\rightarrow\mathbb{R}$ be a real polynomial with zero constant term, denoted by
\begin{equation}\label{polys}
	P_i(t)=\sum_{\sigma_i\leq \beta\leq d_i} a_{i,\beta}t^{\beta},
\end{equation}
where $a_{i,\sigma_i}, a_{i,d_i}$ are nonzero, $1\leq \sigma_i\leq d_i$. We also let $\gamma: I\rightarrow \mathbb{R}^n$ be a polynomial curve defined by
\begin{equation}\label{gamma}
	\gamma(t)=(P_1(t),\ldots,P_n(t)),
\end{equation}
where $I$ is an interval with nonempty interior.
The main result of this paper can be stated as follows:
\begin{theorem}\label{thm1}
	Let $n\geq1$ be an integer, and let $P_i(t)$, $\gamma(t)$ be defined by \eqref{polys}, \eqref{gamma} respectively. Then for any $\epsilon\in (0,1/2)$, there exists a constant $c>0$ only depending on $\gamma$ such that for all $E\subset[0,1]^n$ with $|E|\geq \epsilon$, there exist 
	\[x,x+\gamma(t)\in E\]
	with $t>\delta$, where
	\begin{equation}\label{delta}
		\delta=\delta(\epsilon)=\exp\left(-\exp\left(c\epsilon^{-4}\right)\right).
	\end{equation}
	
\end{theorem}
\begin{remarks}
	Note that when $n=1$, the proof of  Theorem \ref{thm1} is trivial since the equation $P(t)=c\epsilon$ is solvable in $\mathbb{R}$ with $P(0)=0$ and sufficiently small $c$.  When $n=2$ and $P_1,P_2$ are linearly independent, as pointed out in \cite{kos23} and \cite{dks23}, Theorem \ref{thm1} can be derived by the corner-type Roth theorem proven in \cite{cg23}. Indeed, we have the observation that
	\begin{equation*}
		(x,y+P_2(t))=(x-P_1(t),y)+(P_1(t),P_2(t)),
	\end{equation*}
	and in \cite{cg23}, we showed the existence of triple
	$$(x,y),(x-P_1(t),y),(x,y+P_2(t)) $$
	in $E$ for some $t>\delta(\epsilon)=\exp\left(-\exp\left(c\epsilon^{-6}\right)\right)$ which is a little worse than the bound \eqref{delta}. The novelty of this paper is that we get a better bound of $\delta(\epsilon)$ for arbitrary polynomial curves in any dimension.
\end{remarks}

\begin{remarks}
	We prove the theorem for general  real polynomial curves $\gamma$ only requiring that $P_i(0)=0$ for $1\leq i\leq n $. When the $P_i$'s are linearly independent, Theorem \ref{thm1} can be obtained by Bourgain's reduction and the decay estimate of a certain oscillatory integral operator. When these polynomials are linearly dependent, the proof relies on two basic geometric observations:
	\begin{itemize}
		\item[(1)] If the polynomials $\{P_i\}_{1\leq i\leq n} $ are linearly dependent, the curve $\gamma(t)$ defined by \eqref{gamma} lies in a lower dimensional subspace;
		\item[(2)]  If $E\subset [0,1]^n$ has positive Lebesgue measure, i.e. $|E|\geq\epsilon$, and $V\subset\mathbb{R}^n$ is a linear subspace of dimension $k$,  then there is a point $x\in\mathbb{R}^n$ such that the $k$-dimensional Hausdorff measure $\mathcal{H}^k((V+x)\cap E)\gtrsim \epsilon$.
	\end{itemize}
	Then we can use the results of Theorem \ref{thm1} proven in the linear independent case and  lower dimensions. Note that the second observation may not work, and certain restrictions in the curvature of $\gamma(t)$ are crucial in the proof of \cite{kos23} and \cite{bp23}.
\end{remarks}

It is an interesting problem whether the similar results of Theorem \ref{thm1} still hold for the  subset in $[0,N]^n$ of positive density with the gap estimate $t>\delta N^{1/d}$. When $n=1$, the question is trivial since it is equivalent to searching solutions of the equation $P(t)=y-x$ where $y,x\in E$. For general $n$, we also expect a positive result of this question since there is a general theorem proven by Bergelson, Host, McCutcheon and Parreau \cite[Corollary 3.8]{bhmp00} using ergodic theory. They proved that for given $\epsilon>0$, $k,n\in\mathbb{N}$ and $P_{i,j}(t)\in\mathbb{R}[t]$ with $P_{i,j}(0)=0$, $1\leq i\leq k$, $1\leq j\leq n$, there exists a $\delta=\delta(\epsilon)>0$ such that if $E$ is a measurable subset of $[0,N]^n$ with $N\geq1$ and $|E|\geq\epsilon N^n$, then there exist $x\in\mathbb{R}^n$ and $t\in\mathbb{R}$ with $t>\delta N^{1/d}$ such that
\begin{equation*}
	\{x,x+\textbf{P}_1(t),\ldots,x+\textbf{P}_k(t)\}\subset E,
\end{equation*}
where we write $$\textbf{P}_i(t)=(P_{i,1}(t),\ldots,P_{i,n}(t))$$ and $$d=\max_{1\leq i\leq k, 1\leq j\leq n}\deg P_{i,j}.$$
However, their theorem does not give an quantitative estimate of $\delta(\epsilon)$ as in \eqref{delta}. The second result of our paper is a quantitative version of this theorem in the special case where $k=1$ and the related polynomials are of distinct degrees.

\begin{theorem}\label{thm2}
	Let $n\geq1$ be an integer, and let $P_i(t)$, $\gamma(t)$ be defined by \eqref{polys}, \eqref{gamma} respectively. If we additionally assume that the polynomials have distinct degrees and $d=\max_{1\leq i\leq n} d_i$, then for any $\epsilon\in (0,1/2)$ there exists a $\delta=\delta(\epsilon,\gamma)$ defined the same as in \eqref{delta} for some constant $c>0$ only depending on $\gamma$ such that for all $E\subset[0,N]^n$ with $|E|\geq \epsilon N^n$, there exist 
	\[x,x+\gamma(t)\in E\]
	with $t>\delta N^{1/d}$.
	
\end{theorem}

Note that Theorem \ref{thm2} cannot be  derived directly from Theorem \ref{thm1} by simply rescaling arguments. We will prove a stronger result instead and then an anisotropic rescaling is allowed in the proof to eliminate the dependence of the decay estimate on $N$.  Applying this trick, we can also give an extension of the corner-type Roth theorems proven in \cite[Theorem 1.1]{cg23}.

\begin{theorem}\label{thm3}
	Let $P_1,P_2:\mathbb{R}\rightarrow\mathbb{R}$ be two linearly independent polynomials with zero constant term and $\deg P_1<\deg P_2=:d$. Then for any $\epsilon\in(0,1/2)$, there exists a $\delta=\delta(\epsilon)$ with
	\begin{equation}\label{delta-2}
		\delta(\epsilon)=\exp(-\exp(c\epsilon^{-6}))
	\end{equation}
	for some constant $c>0$ only depending on $P_1,P_2$ such that, given any measurable set $S\subset [0,N]^2$ with measure $|S|\geq \epsilon N^2$, it contains a triplet
	\begin{equation*}
		(x,y),(x+P_1(t),y),(x,y+P_2(t))
	\end{equation*}
	with $t\geq \delta N^{1/d}$.
\end{theorem}
In \cite{cdr21}, an observation was noted that if $E\subset[0,N]^2$ with $|E|\geq \epsilon N^2$ and $\widetilde{E}=\{(x,y)\in[0,N]^2: x-y\in E\}$, we have $|\widetilde{E}|\geq \epsilon N^2$. Then the one-dimensional bipolynomial Roth theorem obtained in \cite{cgl21} can be  regained by Theorem \ref{thm3}.

\begin{remarks}
	It is also an interesting problem whether the gap estimate \eqref{delta} can be improved to $\exp(-\delta^{-C})$ for some constant $C$ depending only on $\gamma$ as in \cite[Theorem 1]{dks23}.
\end{remarks}

\medskip

{\it Notations.} For real $X$ and nonnegative $Y$, we use $X\lesssim Y$ to denote $|X|\leq CY$ for some constant $C$. We write $X\lesssim_p Y$ to indicate that the implicit constant $C$ depends on a parameter $p$. If $X$ is nonnegative, $X\gtrsim Y$ means $Y\lesssim X$. The Landau notation $X=O_p(Y)$ is equivalent to $X\lesssim_p Y$. The notation $X\asymp Y$ means that $X\lesssim Y$ and $Y\lesssim X$. We let $\mathbb{N}_0=\mathbb{N}\cup \{0\}$ and $e(x)=\exp(2\pi i x)$. The Fourier transform of $f$ is $\widehat{f}(\xi)=\mathcal{F}(f)(\xi)=\int_{\mathbb{R}} \! f(x)e(-\xi x) \,\textrm{d}x$. $a \gg (\ll)$ $b$ means $a$ is much greater (less) than $b$.  $\textbf{1}_{E}$ represents the characteristic function of a set $E$. For positive integer $n\in\mathbb{N}$, denote $[n]=\{1,2,\ldots,n\}$.

\section{Preliminaries}\label{sec2}

Throughout this paper,  we let $\rho$ be a nonnegative radial smooth bump function on $\mathbb{R}^n$ which is compactly supported and constant on $[-1,1]^n$. We normalize it such that $\widehat{\rho}(0)=1$ and denote $\rho_\ell(x)=2^{n\ell}\rho(2^{\ell}x)$. Let $\tau$ be a non-negative smooth bump function supported on $[1/2,2]$ with $\int\!\tau\,\mathrm{d}t=1$. Set $\tau_{\ell}(t)=2^\ell\tau(2^\ell t).$

We need  several lemmas below. The first lemma is a generalization of \cite[Lemma 6]{bou88}.   
\begin{lemma}\label{bou}
	For a nonnegative function $f$ supported on $[0,1]^n$, and for $s\geq 1$, $k_1,k_2,\ldots,k_{s-1}\in\mathbb{N}_0$, we have
	\begin{equation}\label{s2-1}
		\int_{\mathbb{R}^n}\!f\prod_{j=1}^{s-1}(f*\rho_{k_j})\,\mathrm{d}x\geq c \left(\int_{\mathbb{R}^n}\!f\, \mathrm{d}x\right)^s
	\end{equation}
	for some $c>0$ depending on $s$ and  $\rho$.
\end{lemma}
In this paper, we only need the case $s=2$.  One can refer \cite{kovac22} for the proof of this lemma.

Let $P_i(t)$, $\gamma(t)$ be defined by \eqref{polys}, \eqref{gamma} respectively. For $s\in\mathbb{Z}$, set $$\gamma_s(t)=(2^{-d_1s}P_1(2^st),2^{-d_2s}P_2(2^st),\ldots,2^{-d_ns}P_n(2^st)).$$
Define
\begin{equation}\label{s4-7}
	T_{s,\ell}f(x)=\int_{\mathbb{R}}\!f(x+\gamma_s(t))\tau_{\ell}(t)\,\mathrm{d}t.
\end{equation}
We can obtain some estimates of $T_{s,\ell}$ as bellow. 

\begin{lemma}\label{lem3}
	Let $n\geq1$ be an integer, and let $P_i(t)$ and $\gamma(t)$ be defined by \eqref{polys} and \eqref{gamma} respectively. If we additionally assume that the polynomials have distinct degrees and $d=\max_{1\leq i\leq n} d_i$, then there exist  a sufficiently large integer $\Gamma\gg 1$ and a constant $\mathfrak{b}>0$ only depending on $\gamma$  such that for any $s\in\Gamma(2\mathbb{N}_0)$, $\ell\in \Gamma(\mathbb{N}_0\setminus2\mathbb{N}_0)$\footnote{Here $\Gamma(2\mathbb{N}_0)=\{0,2\Gamma,4\Gamma,\ldots\}$ and $\Gamma(\mathbb{N}_0\setminus2\mathbb{N}_0)=\{\Gamma,3\Gamma,5\Gamma,\ldots\}$.} it follows that
	\begin{equation}\label{s2-20}
		\|T_{s,\ell}f\|_2\lesssim2^{\mathfrak{b}\ell} 2^{-k/d}\|f\|_2
	\end{equation}
	for all measurable functions $f$ on $\mathbb{R}^n$ such that  $\supp \widehat{f}\subset\{\xi\in\mathbb{R}:2^k\leq|\xi|< 2^{k+1}\}$.
\end{lemma}
\begin{proof}
	Let $\psi\in C_0^\infty(\mathbb{R})$ be a bump function supported on $[1/2,4]$ and be equal to 1 on $[1,2]$.
	By the Fourier inversion theorem, we can write $T_{s,\ell}f$ as
	\begin{equation}\label{s4-24}
		\int_{\mathbb{R}^n}\!e^{2\pi i x\cdot\xi}\widehat{f}(\xi)m_{k,s,\ell}(2^{-k}\xi)\,\mathrm{d}\xi,
	\end{equation}
	with
	\begin{equation}
		m_{k,s,\ell}(\xi)=\psi(|\xi|)\int_{\mathbb{R}}\!e^{2\pi i2^k\gamma_s(t)\cdot\xi}\tau_\ell(t)\,\mathrm{d}t\label{s2-27}
	\end{equation}
	We will show that 
	\begin{equation}\label{s2-23}
		\|m_{k,s,\ell}\|_\infty\lesssim 2^{c\ell-k/d}
	\end{equation}
	holds for any $s\in\Gamma(2\mathbb{N}_0)$, $\ell\in \Gamma(\mathbb{N}_0\setminus2\mathbb{N}_0)$.
	Applying the Plancherel theorem to \eqref{s4-24} and using the estimate \eqref{s2-23}, we obtain that
	\begin{equation}
		\|T_{s,\ell}f\|_2=\|m_{k,s,\ell}(2^{-k}\cdot)\widehat{f}\|_2\lesssim  2^{c\ell}2^{-k/d}\|f\|_2.
	\end{equation}
	
	It remains to prove \eqref{s2-23}. For fixed $\xi$ with $|\xi|\in\supp \psi\subset [1/2,4]$,  we have 
	\begin{equation}\label{s2-30}
		|\xi_i|\leq 2, \text{ for $1\leq i\leq n$}.
	\end{equation}
	And by the pigeonhole principle, there exists $1\leq i_0\leq n$ such that
	\begin{equation}\label{s2-21}
		|\xi_{i_0}|\geq \frac{1}{\sqrt{2n}}.
	\end{equation}	  
	We  can rewrite \eqref{s2-27} as
	\begin{equation}
		\psi(|\xi|)\int_{\mathbb{R}}\!e^{2\pi i\lambda\phi_{s,\ell}(t,\xi)}\tau(t)\,\mathrm{d}t,
	\end{equation}
	where $\lambda=2^{k-d\ell}$ and the phase function $\phi_{s,\ell}(t,\xi)$ has the form
	\begin{equation}\label{s2-28}
		\phi_{s,\ell}(t,\xi)=\sum_{1\leq\beta\leq d}\left(\sum_{\substack{1\leq i\leq n\\\sigma_i\leq\beta\leq d_i}}a_{i,\beta}2^{-sd_i}\xi_i\right)2^{s\beta+(d-\beta)l}t^{\beta}.
	\end{equation}
	Note that for any $\alpha\in\mathbb{N}_0$, we have
	\begin{equation}\label{derivative}
		\left|\frac{\partial^{\alpha}\phi_{s,\ell}}{\partial t^{\alpha}}(t,\xi)\right|\lesssim 2^{c\ell}
	\end{equation}
	for some $c\geq0$.
	
	Since $P_i(t)$'s have distinct degrees and changing the order of the axes is not important to the final result, we may without loss of generality assume that $1\leq d_1<d_2<\ldots<d_n=d$. When $s\in\Gamma(2\mathbb{N}_0)\setminus\{0\}$, combining \eqref{s2-28}, \eqref{s2-30} and \eqref{s2-21},  we can compute directly to show that
	\begin{equation}\label{derivative-1}
		\left|\frac{\partial^{d_{i_0}}\phi_{s,\ell}}{\partial t^{d_{i_0}}}(t,\xi)\right|=d_{i_0}!\cdot 2^{(d-d_{i_0})\ell}|\xi_{i_0}|\left(1+o(1)\right)\geq\frac{1}{2\sqrt{2n}}\asymp 1
	\end{equation}
	for any $t\in\supp \tau$ and   $\ell\in \Gamma_0(\mathbb{N}_0\setminus2\mathbb{N}_0)$. By \eqref{derivative}, \eqref{derivative-1} and integration by parts or the method of stationary phase, we get
	\begin{equation*}
		|m_{k,s,\ell}(\xi)|\lesssim 2^{c'\ell}\min\{1, \lambda^{-1/d_{i_0}}\}\leq 2^{c'\ell}\lambda^{-1/d}=2^{c''\ell-k/d}.
	\end{equation*}
	When $s=0$, we can actually derive \eqref{s2-23} for all linearly independent polynomials. We will show this result in Lemma \eqref{lem2} below.
\end{proof}

\begin{lemma}\label{lem2}
	Let $n\geq1$ be an integer, and let $P_i(t)$ and $\gamma(t)$ be defined by \eqref{polys} and \eqref{gamma} respectively. If we additionally assume that the polynomials are linearly independent, then there exist  a sufficiently large $\Gamma\gg 1$ and a constant $\mathfrak{b}\geq 0$ only depending on $\gamma$  such that for any $\ell\geq \Gamma$ it follows that
	\begin{equation}\label{s4-2}
		\|T_{0,\ell}f\|_2\lesssim2^{\mathfrak{b}\ell} 2^{-k/d}\|f\|_2
	\end{equation}
	for all measurable functions $f$ on $\mathbb{R}^n$ such that  $\supp \widehat{f}\subset\{\xi\in\mathbb{R}:2^k\leq|\xi|< 2^{k+1}\}$.
\end{lemma}
\begin{proof}
	Plugging $s=0$ into \eqref{s2-27} gives
	\begin{align}
		m_{k,0,\ell}(\xi)&=\psi(|\xi|)\int_{\mathbb{R}}\!e^{2\pi i2^k\gamma(t)\cdot\xi}\tau_\ell(t)\,\mathrm{d}t\nonumber\\
		&=\psi(|\xi|)\int_{\mathbb{R}}\!e^{2\pi i\lambda\phi_\ell(t,\xi)}\tau(t)\,\mathrm{d}t,\label{s2-7}
	\end{align}
	where $\lambda=2^{k-d\ell}$ and the phase function
	\begin{equation}\label{s2-8}
		\phi_\ell(t,\xi)=2^{d\ell}\gamma(2^{-\ell}t)\xi,
	\end{equation}
	and 
	\begin{equation}\label{derivative-2}
		\left|\frac{\partial^{\alpha}\phi_{\ell}}{\partial t^{\alpha}}(t,\xi)\right|\lesssim 2^{c\ell}
	\end{equation}
	for any $\alpha\in\mathbb{N}_0$ and some constant $c\geq0$.
	It is suffices to show that 
	\begin{equation}\label{s2-13}
		\|m_{k,0,\ell}\|_\infty\lesssim \lambda^{-1/d}=2^{\ell-k/d}.
	\end{equation}
	
	We first consider a toy model $$\gamma(t)=(t^{d_1},t^{d_2},\ldots,t^{d_n}),$$ where $d_1,\ldots,d_n$ are distinct positive integers with $1\leq d_1<d_2<\ldots<d_n=d$.
	Then the phase function
	\begin{equation}\label{s2-9}
		\phi_\ell(t,\xi)=2^{(d-\alpha_1)\ell}t^{\alpha_1}\xi_1+\ldots2^{(d-\alpha_n)\ell}t^{\alpha_n}\xi_n. 
	\end{equation}
	For fixed $\xi$ with $|\xi|\in\supp \psi\subset [1/2,4]$,  we have 
	\begin{equation}\label{s2-10}
		|\xi_i|\leq 2, \text{ for $1\leq i\leq n$}.
	\end{equation}
	And by the pigeonhole principle, there exists $1\leq i_0\leq n$ such that
	\begin{equation}\label{s2-11}
		|\xi_{i_0}|\geq \frac{1}{\sqrt{2n}}.
	\end{equation}	 
	Combining \eqref{s2-9}, \eqref{s2-10} and \eqref{s2-11},  we have 
	\begin{equation}\label{derivative-3}
		\left|\frac{\partial^{d_{i_0}}\phi_\ell}{\partial t^{d_{i_0}}}(t,\xi)\right|=d_{i_0}!\cdot 2^{(d-d_{i_0})\ell}|\xi_{i_0}|\left(1+o(1)\right)\geq\frac{1}{2\sqrt{2n}}\asymp 1,
	\end{equation}
	for $t\in\supp \tau$ and  sufficiently large $\ell$ depending on $\gamma$. By \eqref{derivative-2}, \eqref{derivative-3} and integration by parts or the method of stationary phase, we get
	\begin{equation*}
		|m_{k,0,\ell}(\xi)|\lesssim 2^{c'\ell}\min\{1, \lambda^{-1/d_{i_0}}\}\leq 2^{c'\ell}\lambda^{-1/d}=2^{c''\ell-k/d}.
	\end{equation*}

	For the general case, let \[\gamma(t)=(P_1(t),\ldots,P_n(t)),\]
	where the $P_i(t)$'s, $1\leq i\leq n $, are linearly independent polynomials with $P_i(0)=0$. For $1\leq i\leq n$ we can  write
	\[P_i(t)=\sum_{1\leq\beta\leq d}a_{i,\beta}t^{\beta}\]
	with  $d=\max_{1\leq i\leq n} \deg P_i $ and $a_{i,\beta}=0$ for $\beta<\sigma_i$ or $>d_i$.
	Denote $A=\left(a_{i,j}\right)_{n\times d}$ the  coefficients matrix. Then $A$ has rank $n$ and $\gamma(t)^{T}=A\cdot(t^1,\ldots,t^d)^T$.

	For $\xi\in \supp \psi$, we rewrite the phase function as
	\begin{equation*}
		\phi_\ell(t,\xi)=2^{d\ell}\xi\cdot\gamma(2^{-\ell}t)=\xi A\cdot(2^{(d-1)\ell}t^1,\ldots,t^d)^T.
	\end{equation*}
	Since $\mathrm{rank}(A)=n$, there exists a subset $J\subset [d]$ with $|J|=n$ such that the matrix $\widetilde{A}=\left(a_{i,j}\right)_{i\in J,j\in [n]}$ is invertible. Denote $\widetilde{A}^{-1}$ the inverse matrix of $\widetilde{A}$. Then for any fixed $\xi\in \supp\psi$, we have\footnote{For a matrix $A\in M^{m\times n}$, $\|A\|$ denotes its operator norm if $A$ is viewed as a linear transform from $\mathbb{R}^m$ to $\mathbb{R}^n$ defined by $\xi\mapsto\xi A$.}
	\begin{equation*}
		1/2\leq|\xi|=|\xi\widetilde{A}\widetilde{A}^{-1}|\leq\|\widetilde{A}^{-1}\|\cdot |\xi\widetilde{A}|,
	\end{equation*}
	which means that $|\xi\widetilde{A}|$ has a lower bound $1/(2\|\widetilde{A}^{-1}\|)>0$. On the other hand, $|\xi\widetilde{A}|$ has an upper bound $4\|\widetilde{A}\|$.
	Then \eqref{s2-13} follows by the similar argument as in the monomial case. 
	
\end{proof}

\section{Proof of theorem \ref{thm1} and theorem \ref{thm2}}\label{sec3}

In this section, we will prove Theorem \ref{thm1} and Theorem \ref{thm2} by Bourgain's reduction and the lemmas proven in Section \ref{sec2}. 
\subsection{Proof of Theorem \ref{thm1}}\label{sub-thm1}
We first assume that  \[\gamma(t)=(P_1(t),\ldots,P_n(t)),\]
where the  $P_i$'s are linearly independent polynomials with $P_i(0)=0$.
To prove Theorem \ref{thm1}, it suffices to show that for each $\epsilon>0$ , there exists a $\delta =\delta(\epsilon)$ satisfying \eqref{delta} such that
\begin{equation*}
	I=\int_{[0,1]^n}\int_0^1\!f(x)f(x+\gamma(t))\,\mathrm{d}t\mathrm{d}x\geq \delta
\end{equation*}
holds for all measurable functions $f$ with $\supp(f)\subset [0,1]^n$, $0\leq f\leq 1$ and $\int_{\mathbb{R}^n}\!f\,\mathrm{d}x\geq \epsilon$. Indeed, Theorem \ref{thm1} follows by taking $f=\mathbf{1}_E$.

For any $1\ll\Gamma\leq \ell'\ll \ell\ll \ell''$, we have
\begin{align}
	2^\ell I&\gtrsim \int_{[0,1]^n}\int_0^1\!f(x)f(x+\gamma(t))\tau_\ell(t)\,\mathrm{d}t\mathrm{d}x \nonumber\\
	&=I_1+I_2+I_3,\label{s4-12}
\end{align}
where
\begin{align*}
	I_1=&\int_{[0,1]^n}\int_0^1\!f(x)f*\rho_{\ell'}(x+\gamma(t))\tau_\ell(t)\,\mathrm{d}t\mathrm{d}x \\
	I_2=&\int_{[0,1]^n}\int_0^1\!f(x)(f*\rho_{\ell''}-f*\rho_{\ell'})(x+\gamma(t))\tau_{\ell}(t)\,\mathrm{d}t\mathrm{d}x \\
	I_3=&\int_{[0,1]^n}\int_0^1\!f(x)(f-f*\rho_{\ell''})(x+\gamma(t))\tau_{\ell}(t)\,\mathrm{d}t\mathrm{d}x .
\end{align*}

For the term $I_2$, by the Cauchy-Schwarz inequality, we have
\begin{equation}\label{s4-14}
	|I_2|\leq\|f*\rho_{\ell''}-f*\rho_{\ell'}\|_2.
\end{equation}

For the term $I_3$,  by the Littlewood-Paley decomposition, for $k_0\in\mathbb{Z}$ to be determined later, 
\begin{equation}\label{s4-9}
	I_3=\sum_{k\geq k_0}\int_{[0,1]^n}\int_0^1\!f(x)g_k(x+\gamma(t))\tau_{\ell}(t)\,\mathrm{d}t\mathrm{d}x, 
\end{equation}
where
\begin{equation*}
	\widehat{g_{k_0}}(\xi)=\left(f-f*\rho_{\ell''} \right)^{\wedge}(\xi)\textbf{1}_{\{|\xi|<2^{k_0+1}\}}(\xi),
\end{equation*}
and for $k>k_0$
\begin{equation*}
	\widehat{g_{k}}(\xi)=\left(f-f*\rho_{\ell''} \right)^{\wedge}(\xi)\textbf{1}_{\{2^k\leq |\xi|<2^{k+1}\}}(\xi).
\end{equation*}
By the Cauchy-Schwarz inequality and the mean value theorem, we have
\begin{align}
	&\left|\int_{[0,1]^n}\int_0^1\!f(x)g_{k_0}(x+\gamma(t))\tau_{\ell}(t)\,\mathrm{d}t\mathrm{d}x\right|\nonumber\\
	\leq&\|f\|_2\|\left(f-f*\rho_{\ell''} \right)^{\wedge}\|_{L^2(\{|\xi|<2^{k_0+1}\})}\nonumber\\
	\lesssim& 2^{k_0-\ell''}.\label{s4-10}
\end{align}
For any fixed $k>k_0$, the Cauchy-Schwarz inequality and Lemma \ref{lem2} imply that
\begin{equation}\label{s4-11}
	\left|\int_{[0,1]^n}\int_0^1\!f(x)g_{k}(x+\gamma(t))\tau_{\ell}(t)\,\mathrm{d}t\mathrm{d}x\right|=\int\!f(x)T_{\ell}g_k(x)\,\mathrm{d}x\lesssim 2^{\mathfrak{b}\ell-k/d}.
\end{equation}
By \eqref{s4-9}, \eqref{s4-10}, \eqref{s4-11} and choosing proper $k_0$, we get
\begin{equation}\label{s4-15}
	|I_3|\lesssim 2^{k_0-\ell''}+\sum_{k>k_0}2^{\mathfrak{b}\ell-k/d}\lesssim 2^{\mathfrak{b}_1\ell-\mathfrak{b}_2\ell''}.
\end{equation} 

For the term $I_1$, let 
\begin{align*}
	I_1'=&\int_{[0,1]^n}\int_0^1\!f(x)f*\rho_{\ell'}(x)\tau_{\ell}(t)\,\mathrm{d}t\mathrm{d}x \\
	=&\int_{[0,1]^n}\!f(x)f*\rho_{\ell'}(x)\,\mathrm{d}x. 
\end{align*}
Applying Lemma \ref{bou}  gives that 
\begin{equation}\label{s4-16}
	I_1'\geq  c\left(\int_{[0,1]^n}f \mathrm{d}x\right)^2\geq c\epsilon^2.
\end{equation}
And by the mean value theorem, we get
\begin{align}
	|I_1-I_1'|&\leq \left|\int_{[0,1]^n}\int_0^1\!f(x)\left[f*\rho_{\ell'}(x+\gamma(t))-f*\rho_{\ell'}(x)\right]\tau_{\ell}(t)\,\mathrm{d}t\mathrm{d}x\right|\nonumber\\
	&\leq\int_{[0,1]^n}\int_0^1\left(\int_0^1\!\left|\nabla\left(f*\rho_{\ell'}\right)(x+s\gamma(t))\cdot\gamma(t)\right|\,\mathrm{d}s\right)\tau_{\ell}(t)\,\mathrm{d}t\mathrm{d}x\nonumber\\
	&\lesssim_n 2^{\ell'-\ell}.\label{s4-17}
\end{align}

By \eqref{s4-12},\eqref{s4-14},\eqref{s4-15},\eqref{s4-16} and \eqref{s4-17}, we could conclude that 
\begin{equation*}
	2^{\ell}I+\|f*\rho_{\ell''}-f*\rho_{\ell'}\|_2\geq c\epsilon^2,
\end{equation*}
in particular
\begin{equation}
	2^{\ell''}I+\|f*\rho_{\ell''}-f*\rho_{\ell'}\|_2\geq c\epsilon^2
\end{equation}
for $1\ll\Gamma=\Gamma(n)\leq \ell'\ll \ell\ll \ell''$. If we choose an appropriate sequence $\Gamma=\ell_1<\ell_2<\ldots<\ell_k<\ldots$ (independently of $f$) such that for each $k\in\mathbb{N}$ we have $\ell_{k+1}\asymp C^k\log (\epsilon^{-1})$ and that either 
\begin{equation}\label{s4-5}
	I>2^{-\ell_{k+1}-1}c\epsilon^2
\end{equation}
or 
\begin{equation}\label{s4-6}
	\|f*\rho_{\ell_{k+1}}-f*\rho_{\ell_k}\|_2\geq c\epsilon^2/2.
\end{equation}

By using the Plancherel theorem and the fast decay of $\widehat{\rho}$ we have
\begin{equation}\label{s4-8}
	\sum_{k=1}^{\infty}\|f*\rho_{\ell_{k+1}}-f*\rho_{\ell_k}\|_2^2\leq C_\rho.
\end{equation}
We leave its proof in Appendix \ref{app2}.
Thus the case where \eqref{s4-6} holds can only occur finite times, and then \eqref{s4-5} must holds for some $k=k_0$ with $1\leq k_0\leq K:=\lceil8c^{-2}C_\rho \epsilon^{-4}\rceil+1$. Therefore, we have
\[I>2^{-\ell_{k_0+1}-1}c\epsilon^2\geq2^{-\ell_{K+1}-1}c\epsilon^2. \]
Using this and the estimate of $\ell_k$, we can obtain the expression of $\delta(\epsilon)$.

Next we consider the case where the polynomials are linearly dependent. Without loss of generality we may assume that $\{P_i(t)\}_{1\leq i\leq n_0}$ is a basis of $\{P_i(t)\}_{1\leq i\leq n}$ with $1\leq n_0\leq n-1$.	Then there exists a matrix $L\in M^{n_0\times (n-n_0)}$ such that
\begin{equation*}
	(P_{n_0+1}(t),\ldots,P_n(t))=(P_1(t),\ldots,P_{n_0}(t))L.
\end{equation*}
Consider the linear maps
\begin{align*}
	&\mathcal{L}_1:\mathbb{R}^{n_0}\rightarrow\mathbb{R}^n,\, \bar{x}=(x_1,\ldots,x_{n_0})\mapsto x=(\bar{x},\bar{x}L)\\
	&\mathcal{L}_2:\mathbb{R}^n\rightarrow\mathbb{R}^{n_0},\, x=(x_1,\ldots,x_n)\mapsto \bar{x}=(x_1,\ldots,x_{n_0}).
\end{align*}
Note that  $ V=\textrm{image } \mathcal{L}_1$ is a $n_0$-dimensional linear subspace of $\mathbb{R}^n$, the map $\mathcal{L}_1:\mathbb{R}^{n_0}\rightarrow V$
is an isomorphism and $\mathcal{L}_1\circ\mathcal{L}_2\arrowvert_{V}=id_ V$. For $E\subset [0,1]^n$ and $|E|\geq\epsilon$, we claim that there exists a point $x_0\in\mathbb{R}^n$ such that 
\begin{equation*}
	\mathcal{H}^{n_0}((x_0+ V)\cap S)\gtrsim \epsilon,
\end{equation*}
where $\mathcal{H}^{n_0}$ denotes the Hausdorff measure of dimension $n_0$ and the implicit constant only depends on the map $\mathcal{L}_1$.
Then the set $\bar{E}:=\mathcal{L}_2((x_0+ V)\cap E)\subset [0,1]^{n_0}$ has Lebesgue measure $|\bar{E}|\gtrsim\epsilon$. By the previous result, we have that there exist 
\begin{equation*}
	\bar{x},\bar{x}+\overline{\gamma(t)}\in\bar{E}
\end{equation*}
with $t>\delta(\epsilon)$.
Note that $\mathcal{L}_2\arrowvert_{x_0+ V}:x_0+V\rightarrow \mathbb{R}^{n_0}$ is  bijective. There exist $x,y\in (x_0+ V)\cap E$ such that $\bar{x}=\mathcal{L}_2(x)$ and $\bar{x}+\overline{\gamma(t)}=\mathcal{L}_2(y)$. Then $y-x\in V$ and
\begin{equation*}
	y-x=\mathcal{L}_1\circ\mathcal{L}_2(y-x)=\mathcal{L}_1(\overline{\gamma(t)})=\gamma(t).
\end{equation*}
Therefore, we have $x,x+\gamma(t)\in E$ with $t>\delta(\epsilon)$, which completes the proof.

\subsection{Proof of Theorem \ref{thm2}}\label{sub-thm2}
Without loss of generality we may assume that $1\leq d_1<d_2\ldots<d_n=d$ and $N=2^{sd}$ for some $s\in \Gamma(2\mathbb{N}_0)$ with $\Gamma$ a sufficiently large constant depending only on $\gamma$. 

Observe that for each $E\subset [0,2^{sd}]^n$ with density $\epsilon$, there is a rectangle $R$ of size $2^{sd_1}\times\ldots\times2^{sd_n}$ contained in $[0,2^{sd}]^n$ such that the density of $E$ in $R$ is also greater than $\epsilon$, i.e. $|E\cap R|\geq \epsilon|R|=\epsilon 2^{s(d_1+\ldots+d_n)}$. Therefore the problem is reduced to prove that for any $\epsilon>0$, there exist a constant $c>0$ only depending on $\gamma$ and a $\delta=\delta(\epsilon)$  satisfying \eqref{delta}
such that for all $E\subset[0,2^{sd_1}]\times\ldots\times[0,2^{sd_n}]$ with $|E|\geq \epsilon 2^{s(d_1+\ldots+d_n)}$, there exist 
\[x,x+\gamma(t)\in E\]
with $t>\delta 2^s=\delta N^{1/d}$. Indeed, we only need to show that 
\begin{equation}\label{s6-1}
	\int_{[0,2^{sd_1}]\times\ldots\times[0,2^{sd_n}]}\int_0^{2^s}\!f(x)f(x+\gamma(t))\,\mathrm{d}t\mathrm{d}x\geq \delta2^{s(d_1+\ldots+d_n+1)}
\end{equation}
holds for all measurable functions $f$ with $\supp(f)\subset [0,2^{sd_1}]\times\ldots\times[0,2^{sd_n}]$, $0\leq f\leq 1$ and $\int_{\mathbb{R}^n}\!f\,\mathrm{d}x\geq \epsilon2^{s(d_1+\ldots+d_n)}$. 

Recall that  $$\gamma_s(t)=(2^{-sd_1}P_1(2^st),\ldots,2^{-sd_n}P_n(2^st)).$$ By rescaling, to prove \eqref{s6-1} is equivalent to prove that
\begin{equation}\label{s6-2}
	\int_{[0,1]^n}\int_0^1\!f(x)f(x+\gamma_s(t))\,\mathrm{d}t\mathrm{d}x\geq \delta
\end{equation}
holds for all measurable functions $f$ with $\supp(f)\subset [0,1]^n$, $0\leq f\leq 1$ and $\int_{\mathbb{R}^n}\!f\,\mathrm{d}x\geq \epsilon$.  Similar to the proof of Theorem \ref{thm1}, we can prove Theorem \ref{thm2} by following Bourgain's reduction again and using Lemma \ref{lem3}.

\section{Application: the corner-type Roth theorem in $[0,N]^2$}

For arbitrarily fixed $\ell\in \mathbb{Z}$ with $|\ell|>\Gamma$, we denote
\begin{equation}\label{s7-4}
	\widetilde{P}_{j,\ell}(t)=2^{\mathfrak{r}_j\ell}P_j\left(2^{-\ell}t\right) \quad \text{for $j=1,2,$}
\end{equation}
where we define $\mathfrak{r}_j=\sigma_j$ if $\ell>\Gamma$ and $\mathfrak{r}_j=d_j$ if $\ell<-\Gamma$. Note that when $t\asymp 1$ and $\Gamma$ is large, $\widetilde{P}_{1,\ell}(t)$ and $\widetilde{P}_{2,\ell}(t)$ behave like monomials $a_{1,\mathfrak{r}_1}t^{\mathfrak{r}_1}$ and $a_{2,\mathfrak{r}_2}t^{\mathfrak{r}_2}$ respectively.

Let $\zeta$ be a smooth function with compact support in $\mathbb{R}^2\times[1/2,2]$. Consider a bilinear operator (associated with $\widetilde{P}_1$ and $\widetilde{P}_2$)
\begin{equation}
	\widetilde{T}_l(f_1,f_2)(x,y)=\int_\mathbb{R} \! f_1\left(x+\widetilde{P}_{1,\ell}(t),y\right)f_2 \left(x,y+\widetilde{P}_{2,\ell}(t)\right)\zeta(x,y,t)\,\mathrm{d}t.\label{s7-1}
\end{equation}
In \cite[Theorem 1.6]{cg23}, the first author and Guo proved the following decay estimate.
\begin{theorem}\label{thm-cg}
	Let $P_1$, $P_2$ be two linearly independent polynomials with zero constant term denoted by \eqref{polys} respectively. If $\Gamma$ is sufficiently large (depending only on $P_1,P_2$), then there exist constants $\mathfrak{b}\geq0$ and $\sigma>0$ such that for all $|\ell|>\Gamma$ and $\lambda>1$ we have
	\begin{equation}
		\left\|\widetilde{T}_{\ell}(f_1,f_2)\right\|_1\lesssim 2^{\mathfrak{b}|\ell|}\lambda^{-\sigma}\|f_1\|_2\|f_2\|_2 \label{s7-2}
	\end{equation}
	for all functions $f_1$, $f_2$ on $\mathbb{R}^2$ so that $\widehat{f_j}(\xi_1,\xi_2)$ is supported where $|\xi_j|\asymp\lambda$ for at least one index $j=1,2$. Moreover, if we assume $\mathfrak{r}_1\neq \mathfrak{r}_2$ in addition, then $\mathfrak{b}=0$ and $\sigma$ is an absolute constant.
\end{theorem}

In the rest of this section, to derive Theorem \ref{thm3}, we will modify the proof of \cite[Section 5]{cg23} by following the strategy used in Subsection \ref{sub-thm2}. 

We can still assume that $N=2^{sd}$ with $s\in \Gamma(2\mathbb{N}_0)$. It suffices to show that for any $\epsilon\in(0,1/2)$, there exist a constant $c$ depending only on $P_1,P_2$ and a $\delta=\delta(\epsilon)$ satisfying \eqref{delta-2} such that 
\begin{align*}
	\int_{[0,2^{sd_1}]\times[0,2^{sd_2}]}\int_0^{2^s}\!\!f(x,y)f(x+P_1(t),y)f(x,y+P_2(t))&\,\mathrm{d}t\mathrm{d}x\mathrm{d}y\\
	&>\delta 2^{s(d_1+d_2+1)}.
\end{align*}
for all measurable functions $f$ on $\mathbb{R}^2$ with $\supp(f)\subset[0,2^{sd_1}]\times[0,2^{sd_2}]$, $0\leq f\leq1$ and $\int_{[0,2^{sd_1}]\times[0,2^{sd_2}]}\!f\,\mathrm{d}x\mathrm{d}y\geq \epsilon$.
By changing variables $x\mapsto 2^{sd_1}x$, $y\mapsto 2^{sd_2}y$ and $t\mapsto 2^s t$, the inequality above can be reduced to prove that for all measurable functions $f$ on $\mathbb{R}^2$ with $\supp(f)\subset[0,1]^2$, $0\leq f\leq1$ and $\int_{[0,1]^2}\!f\,\mathrm{d}x\mathrm{d}y\geq \epsilon$, we have
\begin{equation}
	\int_{[0,1]^3}\!f(x,y)f(x+P_{1,s}(t),y)f(x,y+P_{2,s}(t))\,\mathrm{d}t\mathrm{d}x\mathrm{d}y>\delta,
\end{equation}
where
\begin{equation*}
	P_{j,s}(t)=2^{-sd_j}P_j(2^st)
\end{equation*}
for $j=1,2$.

Let 
\begin{equation*}
	I=\int_{[0,1]^3} \!\!  f(x,y)f\left(x+P_{1,s}(t),y\right)f\left(x,y+P_{2,s}(t)\right) \,\mathrm{d}t\mathrm{d}x\mathrm{d}y.
\end{equation*}
For any $\ell',\ell,\ell''\in\Gamma(\mathbb{N}_0\setminus2\mathbb{N}_0)$ with $\ell'<\ell<\ell''$ we have
\begin{align*}
	2^{\ell}I&\gtrsim_{\tau}
	\int_{[0,1]^3} \!\!  f(x,y)f\left(x+P_{1,s}(t),y\right)f\left(x,y+P_{2,s}(t)\right) \tau_{\ell}(t)\,\mathrm{d}t\mathrm{d}x\mathrm{d}y\\
	&=I_1+I_2+I_3,
\end{align*}
where\footnote{For a function $f$ on $\mathbb{R}^2$ and a function $\phi$ on $\mathbb{R}$, partial convolutions are given by $\phi\!*_1\! f(x,y)=\int_{\mathbb{R}}\! f(x-u,y)\phi(u)\,\textrm{d}u$ and $\phi\!*_2\! f(x,y)=\int_{\mathbb{R}}\! f(x,y-u)\phi(u)\,\textrm{d}u$.}
\begin{align*}
	I_1=\!\!\int_{[0,1]^3}\!\!\!  f(x,y)f\left(x+P_{1,s}(t),y\right)&\rho_{\ell'}\!*_2\!f\left(x,y+P_{2,s}(t)\right)\tau_\ell(t) \mathrm{d}t\mathrm{d}x\mathrm{d}y,\\
	I_2=\!\!\int_{[0,1]^3} \!\!\!  f(x,y)f\left(x+P_{1,s}(t),y\right)&(\rho_{\ell''}\!*_2\!f-\rho_{\ell'}\!*_2\!f)\left(x,y+P_{2,s}(t)\right)\\&\quad\quad\tau_\ell(t)\mathrm{d}t\mathrm{d}x\mathrm{d}y,\\
	I_3=\!\!\int_{[0,1]^3} \!\!\!  f(x,y)f\left(x+P_{1,s}(t),y\right)&(f-\rho_{\ell''}\!*_2\!f)\left(x,y+P_{2,s}(t)\right)\tau_\ell(t) \mathrm{d}t\mathrm{d}x\mathrm{d}y.
\end{align*}

By the Cauchy-Schwarz inequality, it is easy to get
\begin{align*}
	|I_2|\leq\|\rho_{\ell''}\!*_2\!f-\rho_{\ell'}\!*_2\!f\|_2.
\end{align*}

To estimate $I_1$, we set
\begin{equation*}
	I_1':=\int_{[0,1]^2} \!\!  f(x,y)\rho_{\ell'}\!*_2\!f\left(x,y\right)\left(\int _\mathbb{R}\! f\left(x+P_{1,s}(t),y\right)\tau_\ell(t) \,\mathrm{d}t\right)\mathrm{d}x\mathrm{d}y.
\end{equation*}
Then by the mean value theorem, we have
\begin{equation*}
	I_1-I_1'=O_{P_2}\left(2^{\ell'-\ell}\right).
\end{equation*}
Notice that the inner integral
\begin{equation}
	\int _\mathbb{R}\! f\left(x+P_{1,s}(t),y\right)\tau_\ell(t) \,\mathrm{d}t \label{s5-23}
\end{equation}
is in fact over $t\asymp2^{-\ell}$. Since $|\ell-s|\geq \Gamma$, the size of $P_{1,s}(t)$ is dominated by its monomial $a_{1,\mathfrak{r}_1}2^{-sd_1}(2^st)^{\mathfrak{r}_1}$, where, from now on, we define $\mathfrak{r}_j=\sigma_j$ if $s<\ell$ and $\mathfrak{r}_j=d_j$ if $s>\ell$. We use the substitution
\begin{equation*}
	\omega=|P_{1,s}(t)|
\end{equation*}
to rewrite the integral \eqref{s5-23} as a convolution. We may assume that $a_{1,\mathfrak{r}_1}<0$ while the case $a_{1,\mathfrak{r}_1}>0$ is the same up to a reflection. Hence
\begin{equation*}
	\eqref{s5-23}=\widetilde{\tau}\!*_1\!f(x,y),
\end{equation*}
where
\begin{equation*}
	\widetilde{\tau}(\omega)=\tau_\ell(t(\omega))t'(\omega).
\end{equation*}
Let $\varsigma_{\mathfrak{r}_1}^{s,\ell}=|a_{1,\mathfrak{r}_1}|2^{-s(d_1-\mathfrak{r}_1)-\mathfrak{r}_1l}$ and $\rho_{\varsigma_{\mathfrak{r}_1}^{s,\ell}}(x)=(\varsigma_{\mathfrak{r}_1}^{s,\ell})^{-1}\rho((\varsigma_{\mathfrak{r}_1}^{s,\ell})^{-1}x)$. Then
\begin{align*}
	\left\|\widetilde{\tau}\!*_1\!f-\rho_{\varsigma_{\mathfrak{r}_1}^{s,\ell'}}\!*_1\!f\right\|_2
	&\leq \left\|\rho_{\varsigma_{\mathfrak{r}_1}^{s,\ell''}}\!*_1\!f-\rho_{\varsigma_{\mathfrak{r}_1}^{s,\ell'}}\!*_1\!f\right\|_2 \\
	&\quad +\left\|\widetilde{\tau}-\widetilde{\tau}*\rho_{\varsigma_{\mathfrak{r}_1}^{s,\ell''}}\right\|_1+ \left\|\widetilde{\tau}*\rho_{\varsigma_{\mathfrak{r}_1}^{s,\ell'}}-\rho_{\varsigma_{\mathfrak{r}_1}^{s,\ell'}}\right\|_1\\
	&=\left\|\rho_{\varsigma_{\mathfrak{r}_1}^{s,\ell''}}\!*_1\!f-\rho_{\varsigma_{\mathfrak{r}_1}^{s,\ell'}}\!*_1\!f\right\|_2+O\left(2^{\ell-\ell''}\right)+O\left(2^{\ell'-\ell}\right).
\end{align*}
The last two bounds follow from rescaling and the mean value theorem. We thus have
\begin{equation*}
	\left|I_1'-I_1''\right|\leq \left\|\rho_{\varsigma_{\mathfrak{r}_1}^{s,\ell''}}\!*_1\!f-\rho_{\varsigma_{\mathfrak{r}_1}^{s,\ell'}}\!*_1\!f\right\|_2+O\left(2^{\ell-\ell''}\right)+O\left(2^{\ell'-\ell}\right),
\end{equation*}
where
\begin{equation*}
	I_1'':=\int_{[0,1]^2} \!\!  f(x,y)\rho_{\ell'}\!*_2\!f\left(x,y\right)\rho_{\varsigma_{\mathfrak{r}_1}^{s,\ell'}}\!*_1\!f(x,y)\,\mathrm{d}x\mathrm{d}y.
\end{equation*}
By  \cite[Lemma 5.1]{cdr21}, an analogue of Bourgain's \cite[Lemma 6]{bou88},
\begin{equation*}
	I_1''\geq c_{\rho} \left(\int_{[0,1]^2} \! f \right)^3\geq c_{\rho}\epsilon^3.
\end{equation*}

To estimate $I_3$, we consider a dyadic decomposition
\begin{equation*}
	f-\rho_{\ell''}\!*_2 f=S_{\lfloor k_0\rfloor}^{(2)}(f-\rho_{\ell''}\!*_2\!f)+\sum_{k>k_0}\Delta_k^{(2)}(f-\rho_{\ell''}\!*_2\!f)
\end{equation*}
with a parameter $k_0>0$ to be chosen below. Then we write $I_3$ as
\begin{equation*}
	I_3=I_4+\sum_{k>k_0}I_{3,k},
\end{equation*}
where
\begin{align*}
	I_4=\!\!\int_{[0,1]^3} \!\!  f(x,y)f\left(x+P_{1,s}(t),y\right)S_{\lfloor k_0\rfloor}^{(2)}&(f-\rho_{\ell''}\!*_2\!f)\left(x,y+P_{2,s}(t)\right)\\
	&\quad\quad\tau_\ell(t) \mathrm{d}t\mathrm{d}x\mathrm{d}y
\end{align*}
and
\begin{align*}
	I_{3,k}=\!\!\int_{[0,1]^3} \!\!  f(x,y)f\left(x+P_{1,s}(t),y\right)&\Delta_{k}^{(2)}(f-\rho_{\ell''}\!*_2\!f)\left(x,y+P_{2,s}(t)\right)\\&\quad\quad\tau_\ell(t) \mathrm{d}t\mathrm{d}x\mathrm{d}y.
\end{align*}
By the Cauchy-Schwarz inequality and the Plancherel theorem, we have
\begin{equation*}
	|I_4|\leq\|f\|_2\left\|S_{\lfloor k_0\rfloor}^{(2)}(f-\rho_{\ell''}\!*_2\!f)\right\|_2\lesssim 2^{k_0-\ell''}.
\end{equation*}
For each $k>k_0$, let $g_k=\Delta_{k}^{(2)}(f-\rho_{\ell''}\!*_2\!f)$. Then we can rewrite
\begin{equation*}
	I_{3,k}=\int_{[0,1]^3} \!\!  f(x,y)f\left(x+P_{1,s}(2^{-\ell}t),y\right)g_k\left(x,y+P_{2,s}(2^{-\ell}t)\right)\tau(t) \mathrm{d}t\mathrm{d}x\mathrm{d}y.
\end{equation*}
Denote $A_j=2^{(d_j-\mathfrak{r}_j)s+\ell\mathfrak{r}_j}$
and
$\widetilde{P}_j(t)=A_jP_{1,s}(2^{-\ell}t)$ for $j=1,2$.  By rescaling and adding a partition of unity we have
\begin{align*}
	|I_{3,k}|\leq &A_1^{-1}A_2^{-1}\cdot\\
	&\sum_{R\in\mathcal{R}_{s,\ell}}\int\!\left|\int\!\widetilde{f}(x+\widetilde{P}_1(t),y)\widetilde{g}_k(x,y+\widetilde{P}_2(t))\zeta_R(x,y)\tau(t)\,\mathrm{d}t\right|\mathrm{d}x\mathrm{d}y,
\end{align*}
where
\begin{equation*}
	\widetilde{f}(x,y)=f\left(A_1^{-1}x,A_2^{-1}y\right),\quad  \widetilde{g}_k(x,y)=g_k\left(A_1^{-1}x,A_2^{-1}y\right),
\end{equation*}
$\mathcal{R}_{s,\ell}$ is the family of almost disjoint unit squares that form a partition of the set $[0,A_1]\times[0,A_2]$ and, for each $R\in\mathcal{R}_{s,\ell}$, $\zeta_R$ is a nonnegative smooth bump function supported in a neighborhood of $R$ such that $\sum_{R\in\mathcal{R}_{s,\ell}}\zeta_R(x,y)=1$ on $[0,A_1]\times[0,A_2]$.

In fact, $\widetilde{P}_j(t)$ is exactly equal to  $\widetilde{P}_{j,s-\ell}(t)$, where $\widetilde{P}_{j,\ell}$ is defined by \eqref{s7-4} for $j=1,2$. Since $|\ell-s|\geq\Gamma$ is sufficiently large, we can apply Theorem \ref{thm3} (with $\lambda=A_2^{-1}2^{k}$) gives that
\begin{align*}
	|I_{3,k}|&\lesssim A_1^{-1}A_2^{-1}\sum_{R\in\mathcal{R}_{s,\ell}}2^{\mathfrak{b}|s-\ell|}2^{-\sigma k}A_2^{\sigma}\left\|\widetilde{f}\right\|_2\left\|\widetilde{g}_k\right\|_2\\
	&\lesssim 2^{\mathfrak{b}|s-\ell|}2^{-\sigma k}A_1A_2^{1+\sigma}.
\end{align*}
Note that if $s<\ell$, $\mathfrak{r}_j=\sigma_j$ and $2^s\leq 2^\ell$; if $s>\ell$, $\mathfrak{r}_j=d_j$ and $\mathfrak{b}=0$. There exists a uniform constant $\mathfrak{b}'>0$ such that $|I_{3,k}|\lesssim 2^{\mathfrak{b}'\ell-\sigma k}$. 
To sum up, by choosing a proper $k_0$, we thus get
\begin{equation*}
	|I_3|\lesssim 2^{k_0-\ell''}+2^{\mathfrak{b}'\ell-\sigma k_0}\lesssim 2^{\mathfrak{b}_1\ell-\mathfrak{b}_2\ell''}
\end{equation*}
for some fixed constants $\mathfrak{b}_1,\mathfrak{b}_2>0$.

After obtaining the estimates of $I_1$, $I_2$ and $I_3$, the proof follows Bourgain's reduction procedure, as displayed in Section \ref{sub-thm1}. We will omit it.

\section*{acknowledgments}
This project was suppported by the National Key R\&D program of China: No.2022YFA1005700 and the NSF of China under grant No.12371095.  The authors would like to thank the associated editor and anonymous referee for their helpful comments and suggestions.

\section*{Appendix A:  Proof of \eqref{s4-8}}\label{app2}
Observe that 
\begin{align*}
	\sum_{k=0}^\infty\|f*\rho_{\ell_k}-f*\rho_{\ell_{k+1}}\|_2^2&=\sum_{k=0}^{\infty}\int\!|\widehat{f}(\xi)|^2|\widehat{\rho}(2^{-\ell_k}\xi)-\widehat{\rho}(2^{-\ell_{k+1}}\xi)|^2\,\mathrm{d}\xi\\
	&=:J_1+J_2+J_3,
\end{align*}
where
\begin{align*}
	J_1=&\sum_{k=0}^{\infty}
	\int_{|\xi|\leq2^{\ell_k/2}}\!|\widehat{f}(\xi)|^2|\widehat{\rho}(2^{-\ell_k}\xi)-\widehat{\rho}(2^{-\ell_{k+1}}\xi)|^2\,\mathrm{d}\xi,\\
	J_2=&\sum_{k=0}^{\infty}
	\int_{2^{\ell_k/2}<|\xi|<2^{\ell_{k+1}/2}}\!|\widehat{f}(\xi)|^2|\widehat{\rho}(2^{-\ell_k}\xi)-\widehat{\rho}(2^{-\ell_{k+1}}\xi)|^2\,\mathrm{d}\xi,\\
	J_3=&\sum_{k=0}^{\infty}
	\int_{|\xi|>2^{\ell_{k+1}/2}}\!|\widehat{f}(\xi)|^2|\widehat{\rho}(2^{-\ell_k}\xi)-\widehat{\rho}(2^{-\ell_{k+1}}\xi)|^2\,\mathrm{d}\xi.
\end{align*}
For the term $J_1$, by the mean value theorem, we have
\begin{align*}
	|J_1|&=\sum_{k=0}^{\infty}
	\int_{|\xi|\leq2^{\ell_k/2}}\!|\widehat{f}(\xi)|^2\left(\left\|\nabla\widehat{\rho}\right\|_\infty2^{-\ell_k}|\xi|\right)^2\,\mathrm{d}\xi\\
	&\leq \left\|\nabla\widehat{\rho}\right\|_\infty^2\|f\|_2^2\sum_{k=0}^{\infty}2^{-\ell_k}\\
	&\lesssim \|f\|_2^2.
\end{align*}
For the term $J_2$, it is trivial that
\begin{equation*}
	|J_2|\leq \|\widehat{\rho}\|_\infty^2\sum_{k=0}^{\infty}
	\int_{2^{\ell_k/2}<|\xi|<2^{\ell_{k+1}/2}}\!|\widehat{f}(\xi)|^2\,\mathrm{d}\xi\lesssim \|f\|_2^2.
\end{equation*}
For the term $J_3$, the fast decay of $\widehat{\rho}$ and the choice of $\ell_k$ imply that
\begin{align*}
	|J_3|&\lesssim \sum_{k=0}^{\infty}
	\int_{|\xi|>2^{\ell_{k+1}/2}}\!|\widehat{f}(\xi)|^2\frac{1}{(1+2^{-\ell_k}|\xi|)^N}\,\mathrm{d}\xi\\
	&\leq\sum_{k=0}^{\infty}\frac{1}{(1+2^{\ell_{k+1}/2-\ell_k})^N}
	\int_{|\xi|>2^{\ell_{k+1}/2}}\!|\widehat{f}(\xi)|^2\,\mathrm{d}\xi\\
	&\lesssim\|f\|_2^2.
\end{align*}
Thus, we have
\begin{equation*}
	\sum_{k=0}^\infty\|f*\rho_{\ell_k}-f*\rho_{\ell_{k+1}}\|_2^2\leq C_{\rho}\|f\|_2^2.
\end{equation*}

\end{document}